\documentclass{amsart}
 
\usepackage{graphicx} 
 
\newtheorem{theorem}{Theorem}[section]

\newtheorem{lemma}[theorem]{Lemma}
\newtheorem{definition}{Definition}[section]
\newtheorem{proposition}{Proposition}[section]

\begin{document}

\title{COLORED UNLINKING} 

\author{NATALIE DUBOIS}

\address{npdubois25@gmail.com} 

\author{CHRIS EUFEMIA}

\address{cbeufemia@gmail.com}

\author{JEFF JOHANNES}

\address{State University of New York at Geneseo\\
1 College Circle\\
Geneseo, NY  14454\\
johannes@geneseo.edu}

\author{JENNA ZOMBACK}

\address{University of Illinois at Urbana-Champaign\\
1409 West Green Street\\
Urbana, IL 61801\\
zomback2@illinois.edu}

\maketitle

\begin{abstract}
In links with two components there are three different types of crossings: self-crossings in the first component, self crossings in the second component, and crossings between components. In this paper we examine the minimum number of crossing changes needed to unlink without changing the crossings between components. We restrict our attention to unlinking two component links with linking number zero and both components unknotted.  We provide data for links with no more than ten crossings and general results about asymmetry of unlinking between components.

\bigskip


\bigskip

\end{abstract}

\normalsize{
\section{Introduction}
The study of unknotting and unlinking reaches back to the infancy of knot theory [1, 2], and even into its prehistory with tales of untying the Gordian knot.  It is also a current 
practical study with natural applications to biological reactions [3, 4].  In particular, for physical links with different components of different materials, there is an important distinction to be made concerning which crossings are being modified.  In this paper we investigate the minimum number of crossing changes necessary to transform a two component link into the unlink. In doing this, we have two options for which type of crossings to change: we may change internal crossings (within one component), or external crossings (between the two components). Previous work by Peter Kohn [5] has not made this distinction when determining the minimum number of crossing changes necessary to produce the unlink. We, on the other hand, will determine the minimum number of \textit{internal} crossings necessary to produce the unlink. Given a two component link $L$, fix one component as the first component and the other as the second component. (In this paper the first component will be in bold in all diagrams.) 
\begin{definition}
We say that $L$ is \textit{(a,b)-unlinkable} if $a$ crossing changes in the first component and $b$ crossing changes in the second component can produce the unlink. \end{definition}
Clearly, if $L$ is $(a,b)$-unlinkable, $L$ is also $(c,d)$-unlinkable for any $c\geq a$ and $d \geq b$. Because of this, we are mostly concerned with the minimum values of $a$ and $b$ such that $L$ is $(a,b)$-unlinkable.
\begin{definition} If $L$ is $(a,b)$-unlinkable, we say that $(a,b)$ is a vector in the \textit{unlinking region} of $L$. \end{definition}

\section{Examples}
Consider the link $L10n57$ from the Thistlethwaite Link Table [6].   
It can be unlinked with only one first-component crossing change (if we change the lowest crossing in the diagram, this becomes the unlink).

\begin{figure}[th]
 \begin{center}\includegraphics[width=1.6in]{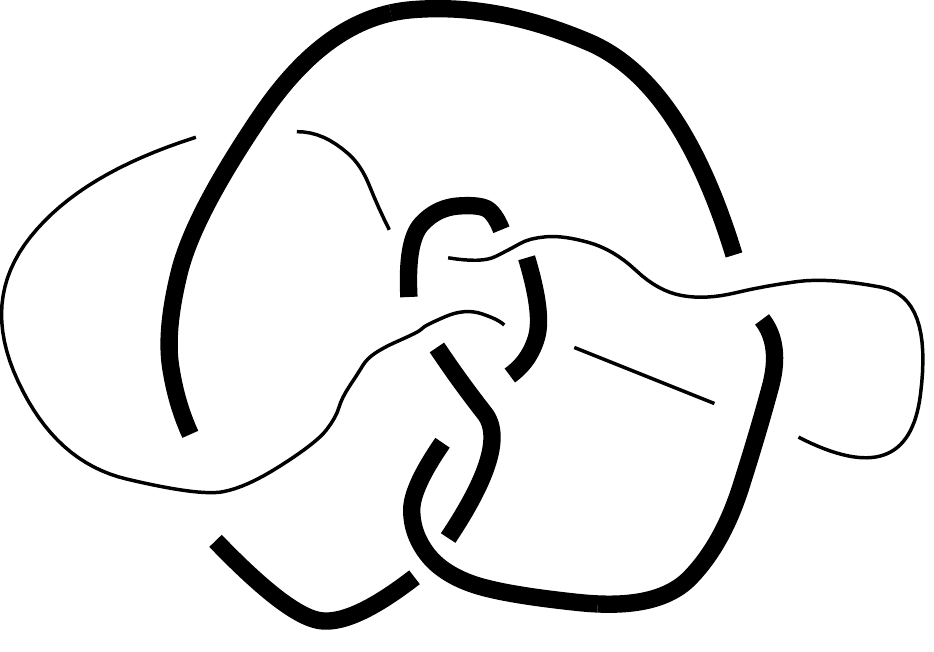}\end{center}
 \vspace*{8pt}
\caption{$L10n57$ \label{fig1}}
\end{figure}
\begin{figure}[th]
\begin{center}  \includegraphics[width=1.6in]{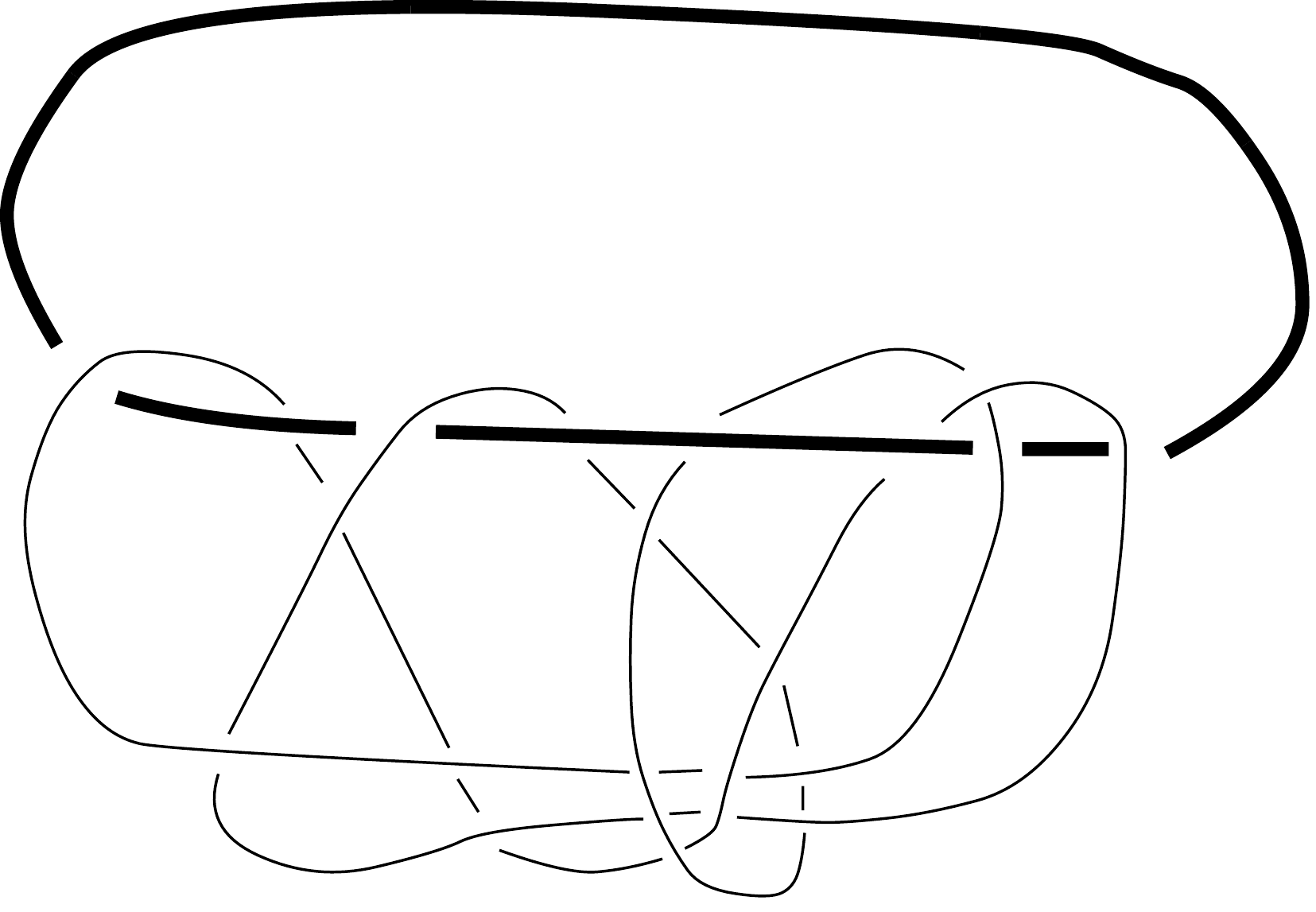} \end{center}
  \vspace*{8pt}
  \caption{$L10n57$ interchanged \label{fig2}}
  \end{figure}
  
  However, if we interchange the link so that the first component is a circle, we need at least four second-component crossing changes to unlink the link.  We do not know whether we can do this in four second-component crossing changes, but we do know that we can unlink this link in five, and that it cannot be done in three.
   
In this case, we might say that the unlinking region is asymmetric; in a sense, the second component is more ``linked" than the first component. After observing several examples like this one, we begin to wonder just how asymmetric these unlinking regions can be. We will prove later that the difference between $(a,0)$ unlinkability and $(0,b)$ unlinkability can be arbitrarily large. 

We want to make sure that we can always unlink $L$ by changing crossings in just one component, leaving the other component as it was. That is, we want to ensure that the links we examine are each $(a,0)$ and $(0,b)$ unlinkable for some positive integers $a$ and $b$. Therefore, we must place two restrictions on our links: that
each component is unknotted, and that the linking number is zero (since changing internal crossings does not affect the linking number). 

For links with fewer than eleven crossings in their minimal diagrams, twenty-two have linking number zero and both components unknotted.  
A table of new results for these links follows.  In this table, the first column gives the Rolfsen [7] and Thistlethwaite [6] codes for the link.  The second column indicates whether the link is interchangeable.  The total unlinking numbers come from [5] and allow changing all types of crossings.  The upper bounds for the colored unlinking values are determined by experiment.  The lower bounds are determined in simple cases, such as $L5a1$ and $L7a4$, by the total unlinking number.  For the case of no crossings changed in one component we use Kohn's unlinking number restricted to a component, the method of computing the linking number of the two components lifted over one component in the two-fold cover branched over the other component. 

The link $L9a18$ provides a simple yet challenging example.  For this link it is apparent that changing three crossings in the standard diagram will unlink.  Because it is interchangeable, these three crossings can be distributed in any way between the two components.  The total unlinking number is 2, so it cannot be unlinked with merely one crossing in either component.  Use of branched covers reveals that the unlinking number restricted to each component is 3.  None of this rules out the apparently unlikely possibility that the link can be undone by changing one crossing within each component. 

The remaining column in the table is the coefficient on $z^3$ in the Conway polynomial for the link. Consider $L9a18$ again.  Using the crossing change formula in [8], we can show that it if were possible to unlink it by changing two self-crossings, we would have $(c_3(9^2_{10}) - c_3(L)) + (c_3(L) - c_3(0^2_1) = \pm k^2 + \pm j^2 = 3$.  Clearly this can only be satisfied  by either $2^2-1^2$ or by $-1^2 + 2^2$.  From this we see that it would need to be done by changing one positive crossing and one negative crossing.  It seems that there are several negative crossings in $L9a18$, so to untie it we must undo negative crossings.  Furthermore, if we change the positive crossing first, we will have an intermediate link, call this $L$.  If we smooth the crossing change between $L9a18$ and $L$ the lobes must each link the other component once (one positive, and one negative).  And, if we smooth the crossing change between $L$ and $L0a1$ the lobes must each link the other component twice (each sign).  Of course, $L$ must be linking number zero and unlinking number one.    None of this analysis suffices to rule out the possibility of one crossing in each component unlinking $L9a18$.  To avoid this subtlety we will henceforth only consider making all changes in one component or the other.  

\begin{table}[ht]
Unlinking data for up to ten crossings.
{\begin{tabular}{|c c| c c c c|}
\hline
Thistlethwaite & Rolfsen & Interchangeable & Total Unlinking & $|c_3|$ & Colored Unlinking\\
\hline 
$L0a1$ & $0^2_1$ & Y & 0 & 0 & (0,0) \\
$L5a1$ & $5^2_1$ & Y & 1 & 1 & (1,0), (0,1)\\
$L7a4$ & $7^2_3$ & Y & 2 & 2 & (2,0), (1,1), (0,2)\\
 $L7a1$ & $7^2_6$ &N & 2 & 1 & (2,0), (1,1), (0,3-4)\\
 $L8a1$ & $8^2_{13}$ &N & 2 &1 & (2,0), (1,1), (0,5-6)\\
$L9a40$ & $9^2_4$ & Y & 2 & 5 & (2,0), (1,1), (0,2)\\
$L9a38$ & $9^2_5$ & Y & 1 & 4 & (1,0), (0,1)\\
 $L9a35$ & $9^2_9$ &Y & 2 & 3 & (2,0), (1,1),(0,2)\\
 $L9a18$ & $9^2_{10}$ &Y & 2 & 3 & (3,0), (2,1),(1,1-2),(0,3)\\
 $L9a1$ & $9^2_{32}$ &N & 2 & 1 & (2,0),(1,1-4),(0,7-9)\\
 $L9a9$ &$9^2_{37}$ & N & 2 & 2 & (3,0),(2,0-1),(1,1-4), 0,6-9)\\
$L9a42$ &$9^2_{41}$ &  N & 2 & 3 & (2,0),(1,1),(0,3-4)\\
 $L10a21$ & & N & 2 & 0 & (2-3,0),(0,20)\\
 $L10a89$ & & Y & 2 & 0 & (2,0),(1,1),(0,2)\\
 $L10a90$ & & Y & 2 & 0 & (2,0),(1,1),(0,2)\\
 $L10a91$ & & Y & 2 & 2 & (2,0),(1,1),(0,2)\\
 $L10a95$ & & Y & 1 & 0 & (1,0),(0,1)\\
 $L10a103$ & & Y & 1 & 0 & (1,0),(0,1)\\
 $L10a111$ & & Y & 2 & 1 & (2,0),(1,1),(0,2)\\
 $L10a112$ & & N & 2 & 1 & (3,0),(1,1),(0,2)\\
 $L10a113$ & &Y & 3 & 1 & (3,0),(2,1),(1,2),(0,3)\\
 $L 10n57$ & & N & 1 & 4 & (1,0),(0,4-5)\\
 \hline
\end{tabular}}
\end{table}

\section{Upper bounds on unlinking number}
The method for determining upper bounds for the unlinking region of a link is conceptually simple; however, in practice it is often difficult to figure out which crossings to change to produce the unlink. The following is an algorithm that will always produce the unlink, and therefore gives an upper bound on $(a,0)$ and $(0,b)$ unlinking numbers. 
\begin{definition} We will say that $L$ is in \textit{parallel strands disk form} if one component is a geometric circle and the other component passes through the disk bounded by this circle in parallel strands, while the rest of the other component lies to the left of the circle. \end{definition}
Notice that any of our links can be put into this form because both of our components are unknotted.  Therefore we may transform one into a circle.  After doing this, putting the intersections into general position makes them discrete inside the circle.  Finally, combing the strands above and below the circle will eliminate crossings on one side and all remaining strands can be gathered together on the left.

To find an upper bound for the $(a,0)$ unlinking number, present it in parallel strands disk form with the second component as a circle. Starting anywhere on the first component, traverse the component, changing crossings as necessary so that depth strictly increases (goes into the page). 

Because the linking number is zero, intersections cannot alternate between the interior and exterior disk, therefore there must be two consecutive intersections for one of the disks.  This produces an innermost arc that may be pulled through after changing the crossings as described above; the crossings do not impede since they are descending into the page.  By repeating this process the link may be split.  Furthermore, the crossing information guarantees each component is unknotted.

Notice that we could equivalently change crossings so that depth into the page strictly decreases. This gives the following:
\begin{proposition}
Given a link in parallel strands disk form with respect to the first component, $(\lfloor\frac{\text{number of first-component crossings}}{2}\rfloor,0)$ is in the unlinking region of $L$. 
\end{proposition}

\section{Main Result}

We mentioned that $L10n57$ is $(1,0)$-unlinkable, but not $(0,3)$-unlinkable. In fact:
\begin{theorem}
Given any positive integer $d$, there exists a link $L$ that is $(a,0)$-unlinkable, but not $(0,a+d)$-unlinkable for some $a \in \mathbb{N}$. 
\end{theorem}
To prove this theorem, we investigate generalizations of link $L8a1$. We generalize the link $L8a1$ as shown below by inserting $n$ full twists in the bottom area of the link. 
\begin{figure}[th]
\begin{center}
\includegraphics[width=4cm]{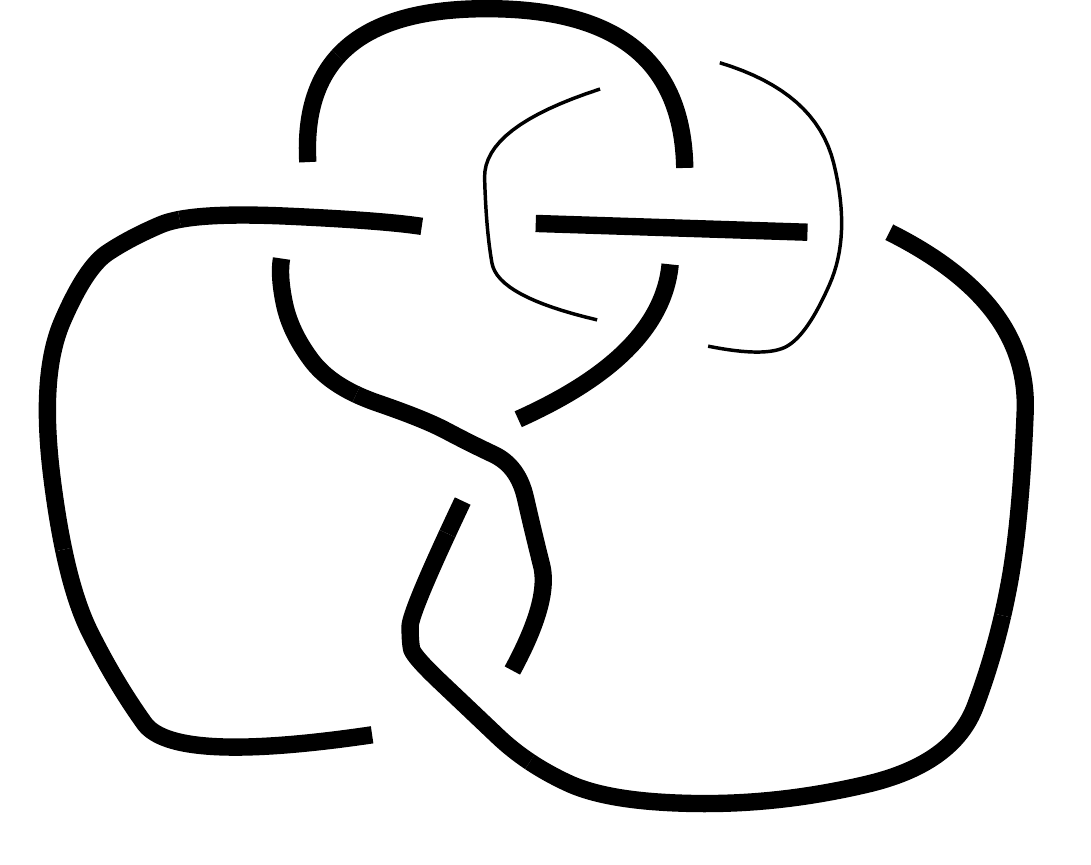}
\end{center}
\vspace*{8pt}
\caption{$L8a1$ \label{fig3}}
\end{figure}
\begin{figure}[th]
\begin{center}
\includegraphics[width=4cm]{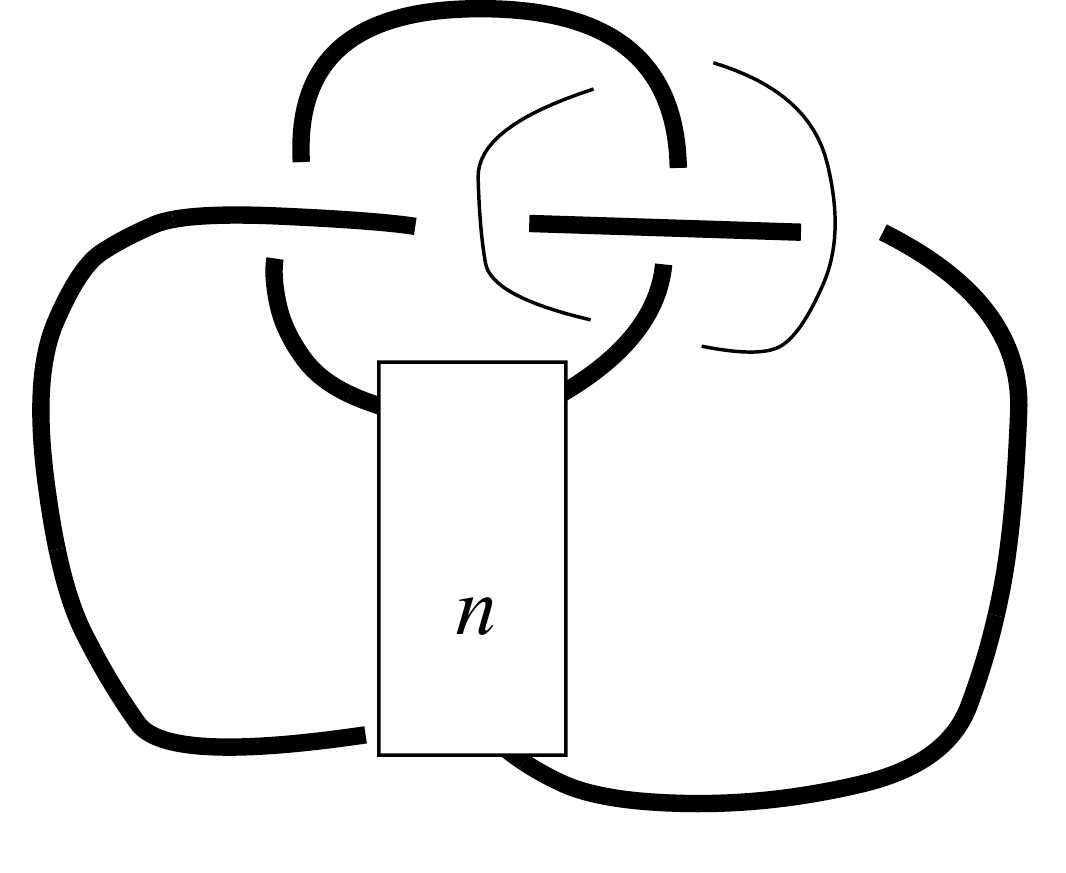} \end{center}
\vspace*{8pt}
\caption{Generalized $L8a1$ \label{fig4}}
\end{figure}
Notice that if we change the first component crossing in the middle of the projected second component, our components split, but the first component will be knotted. If we change the same first component crossing again, we return to the original first component (which was unknotted), so we will have the unlink. These two changes produce the unlink regardless of the value of $n$. So for any value of $n$, $(2,0)$ is in the unlinking region of the $n^{\text{th}}$ generalization of this link. We will investigate the two-fold cover branched over the first component in the $n^{\text{th}}$ generalization to show that as $n$ increases, so does the lower bound of second component crossings we need to change. This will guarantee that there are, in fact, arbitrarily large gaps in the unlinking regions of these links. 

To find the branched cover, we will interchange the $n^{\text{th}}$ member of this family to make the first component a circle (here demonstrated for $L8a1$ with a schematic for the general case):  
\begin{figure}[th]
\begin{center}
\includegraphics[width=5cm]{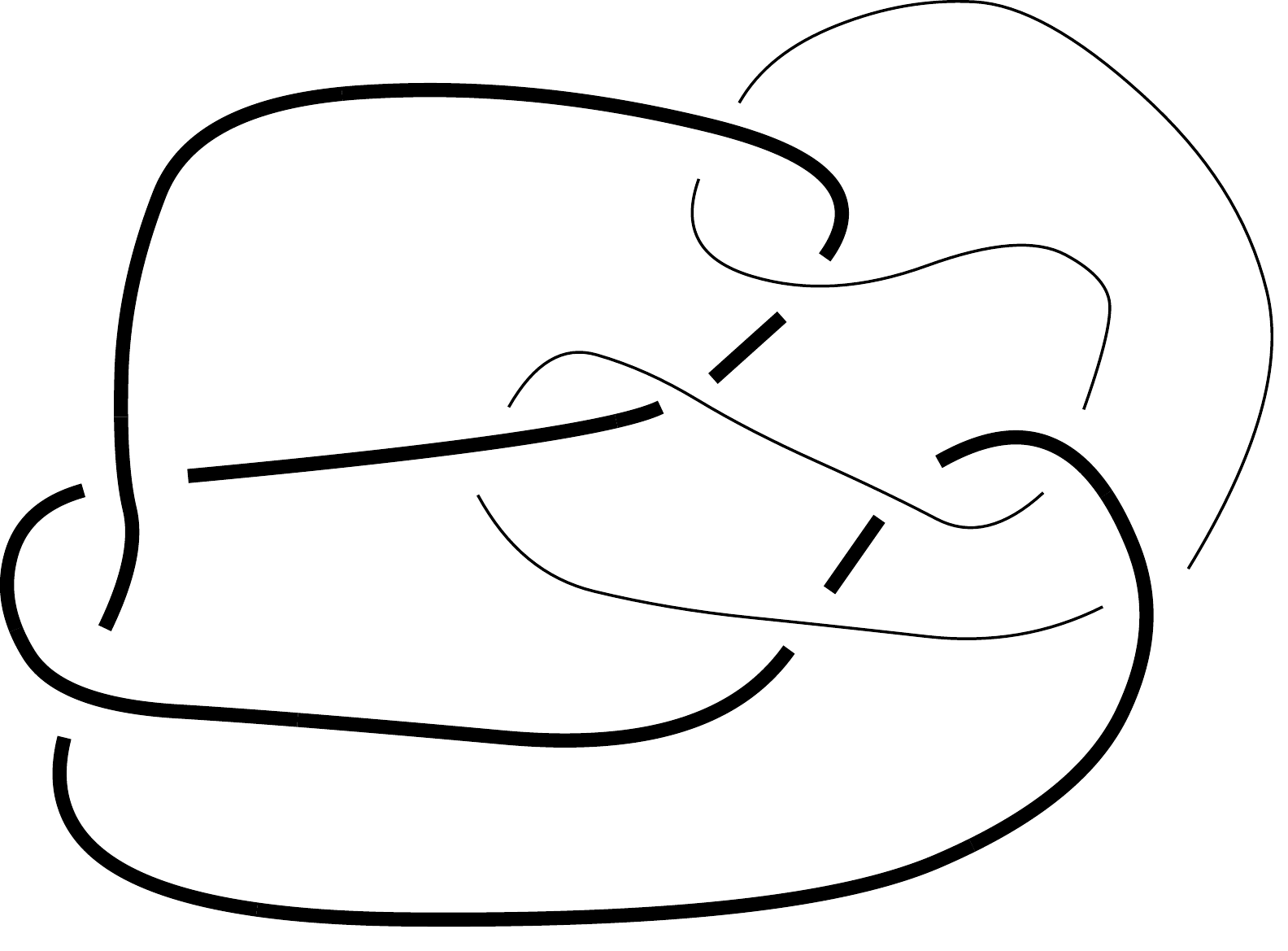}
\end{center}
\vspace*{8pt}
\caption{$L8a1$ interchanged step one \label{fig5}}
\end{figure}
\begin{figure}[th]
\begin{center}
\includegraphics[width=5cm]{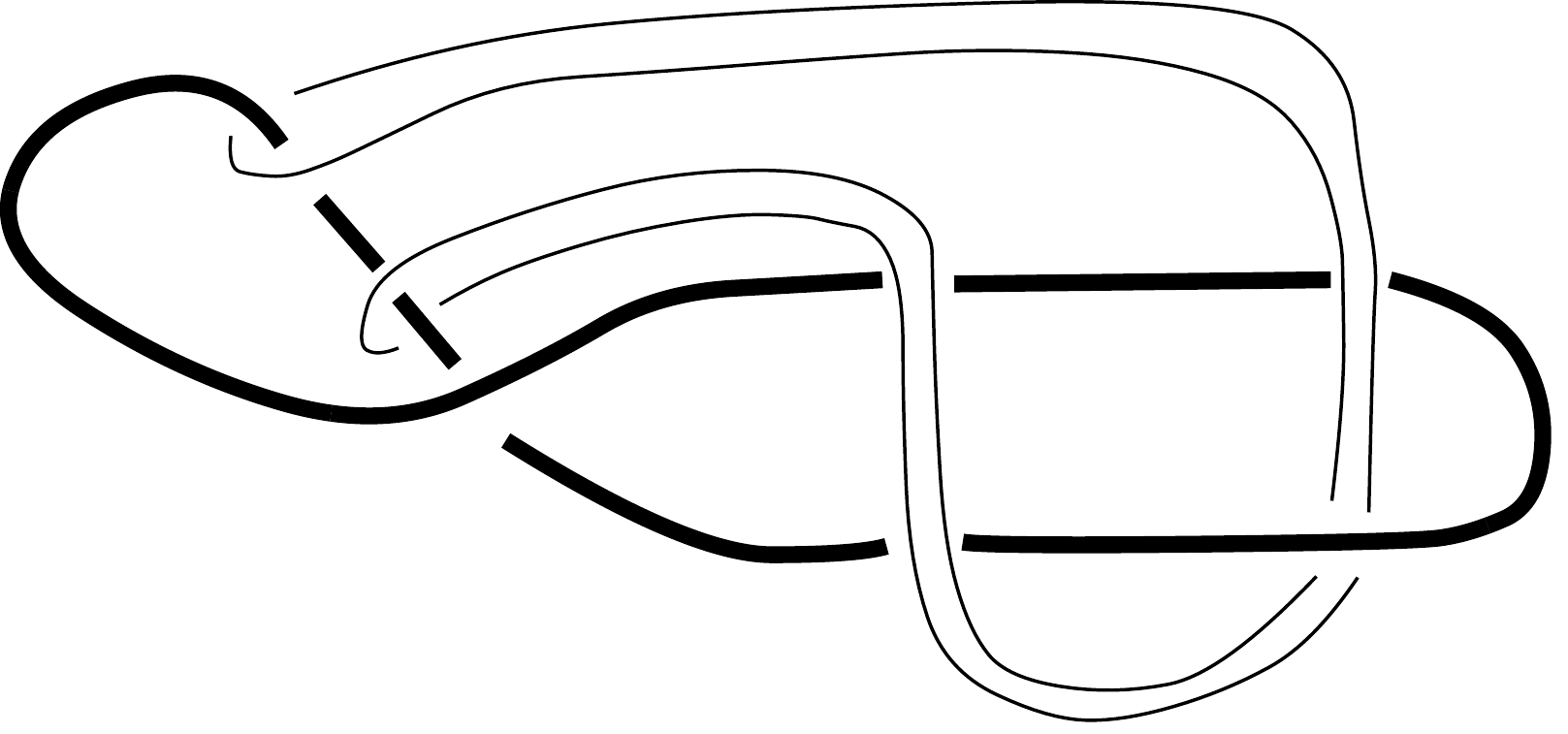}
\end{center}
\vspace*{8pt}
\caption{$L8a1$ interchanged step two \label{fig6}}
\end{figure}
\begin{figure}[th]
\begin{center}
\includegraphics[width=5cm]{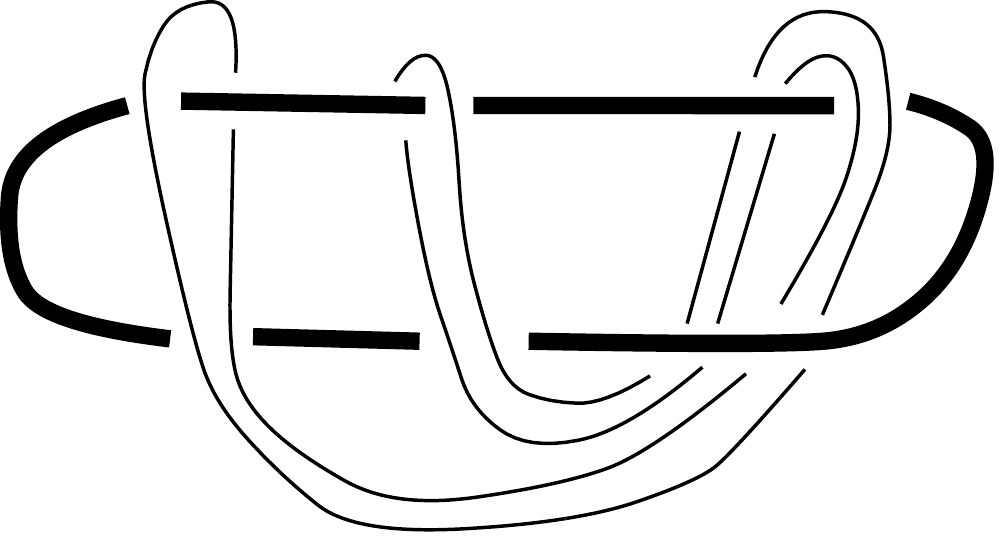}
\end{center}
\vspace*{8pt}
\caption{$L8a1$ interchanged complete \label{fig7}}
\end{figure}
\begin{figure}[th]
\begin{center}
\includegraphics[width=5cm]{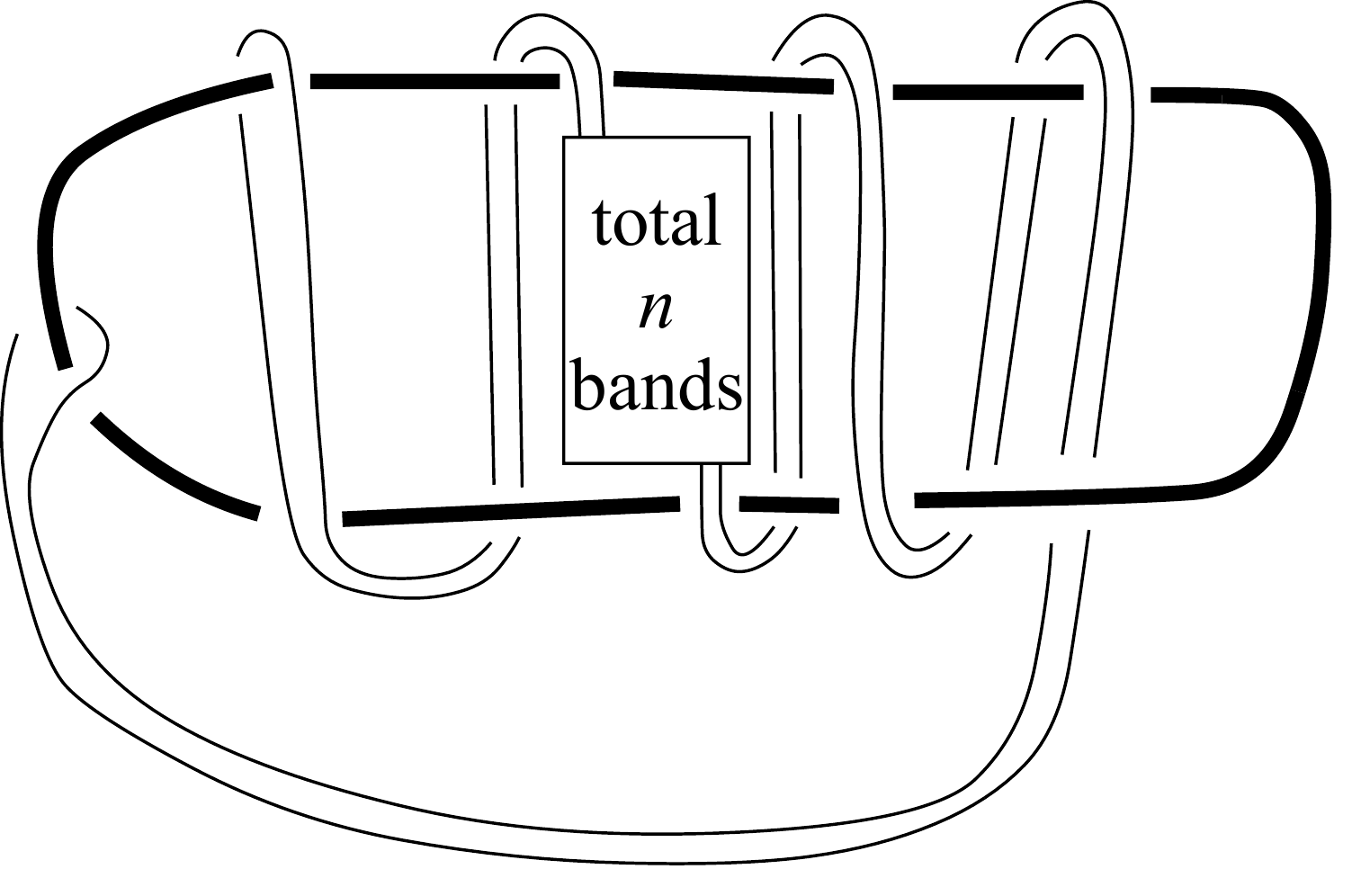}
\end{center}
\vspace*{8pt}
\caption{$L8a1$ generalization interchanged  \label{fig8}}
\end{figure}
\newpage 
\begin{lemma}
The $n^{th}$ member of the $L8a1$ family is not $(0,2n)$-unlinkable. 
\end{lemma}
\begin{proof}
The following sequence of diagrams shows how to obtain a diagram of the $n^{\text{th}}$ member of the $L8a1$ family in parallel strands disk form with the first component a circle. 
One by one, we will slide the loops to the right of the first circle and below it, starting with the left-most loop. We will stop when we reach the first loop.
We now slide the first loop to the left of the diagram, and align the strands through the first component.

\begin{figure}[th]
\begin{center}
\includegraphics[width=5cm]{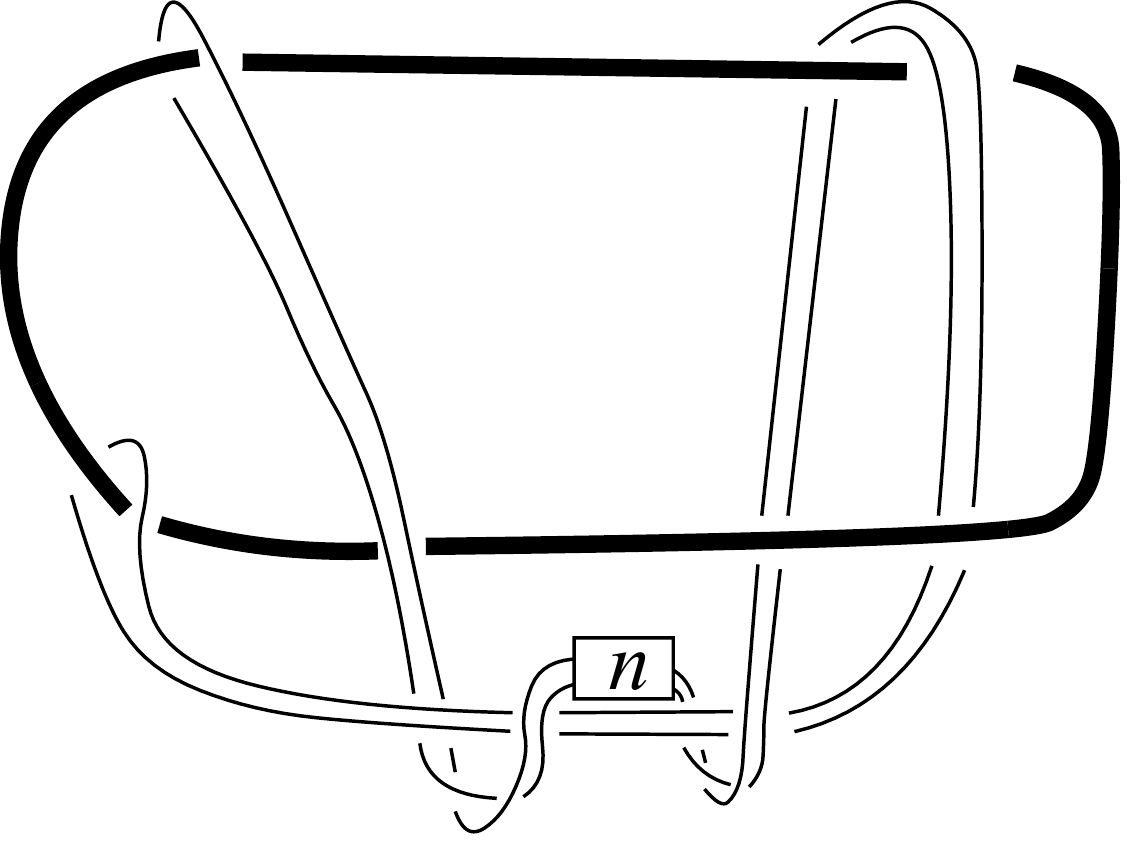}
\end{center}
\vspace*{8pt}
\caption{Parallel strands disk form step one \label{fig9}}
\end{figure}

\begin{figure}[th]
\begin{center}
\includegraphics[width=5cm]{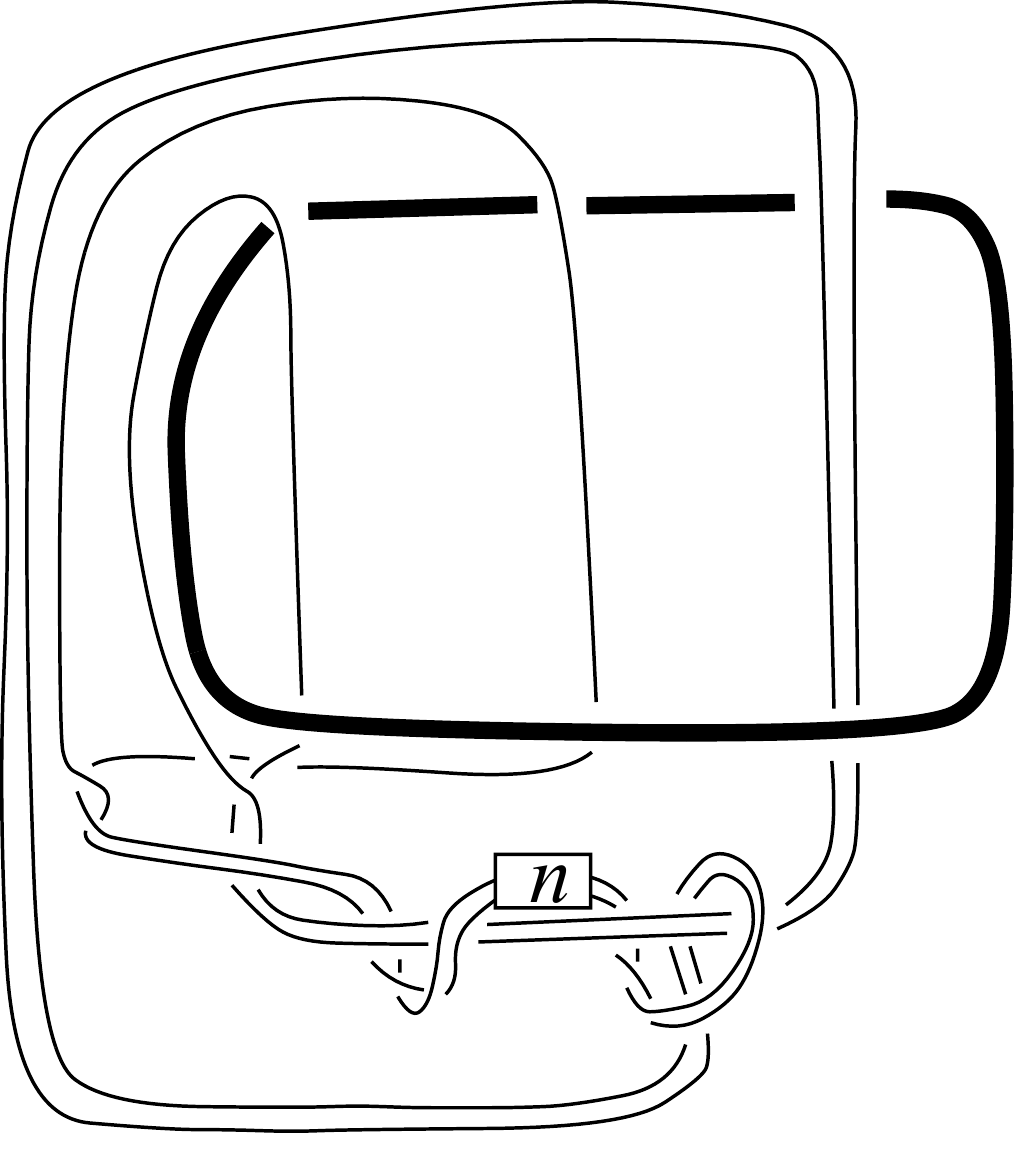}
\end{center}
\vspace*{8pt}
\caption{Parallel strands disk form step two \label{fig10}}
\end{figure}

Finally, we move all $2n+1$ half twists to the same portion of the diagram. The resulting diagram is in parallel strands disk form. 

\begin{figure}[th]
\begin{center}
\includegraphics[width=5cm]{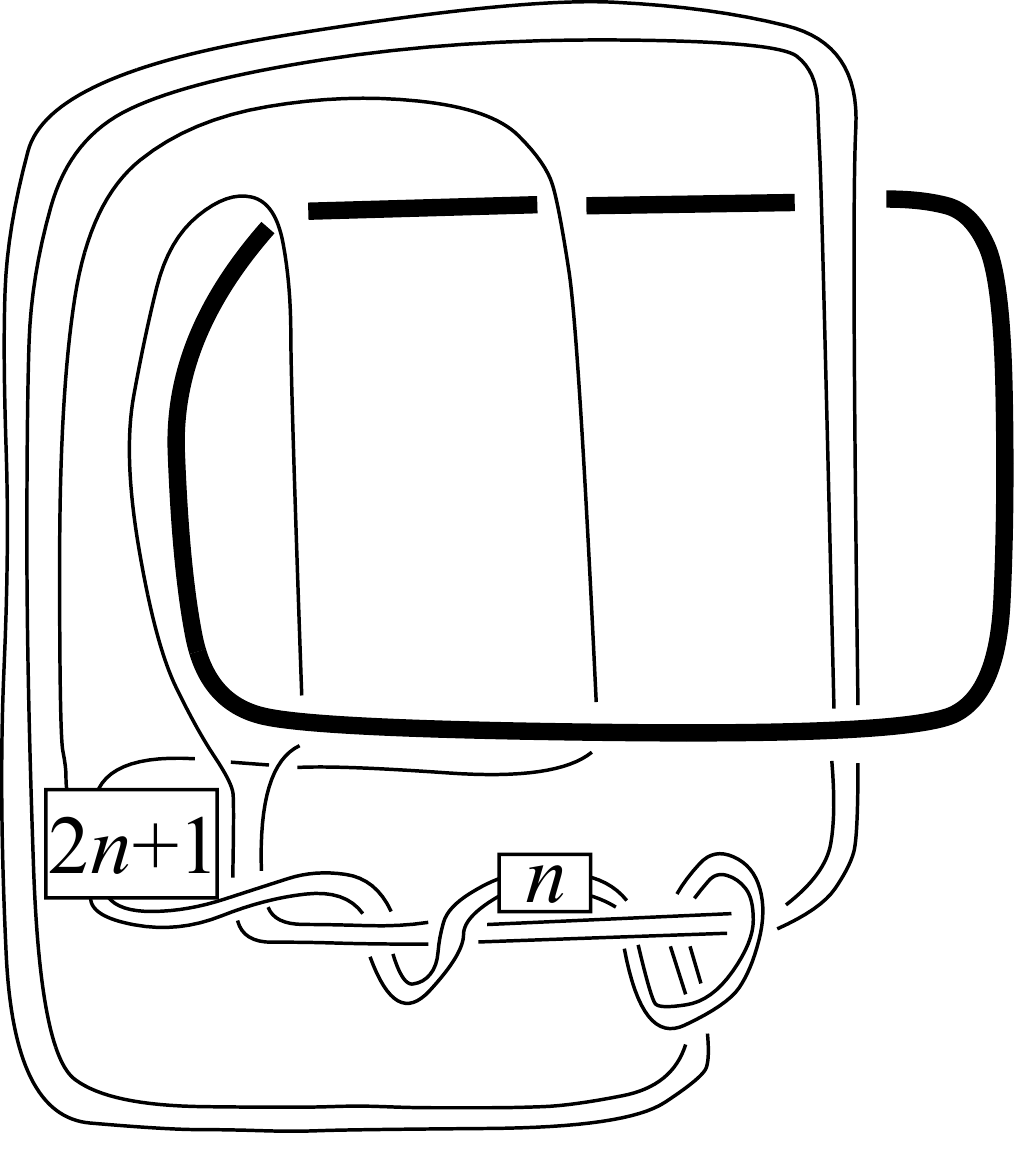}
\end{center}
\vspace*{8pt}
\caption{Parallel strands disk form complete \label{fig11}}
\end{figure}
We now discover the following two-fold branched cover of this link branched over the bold component. 
\begin{figure}[th]
\begin{center}
\includegraphics[width=10cm]{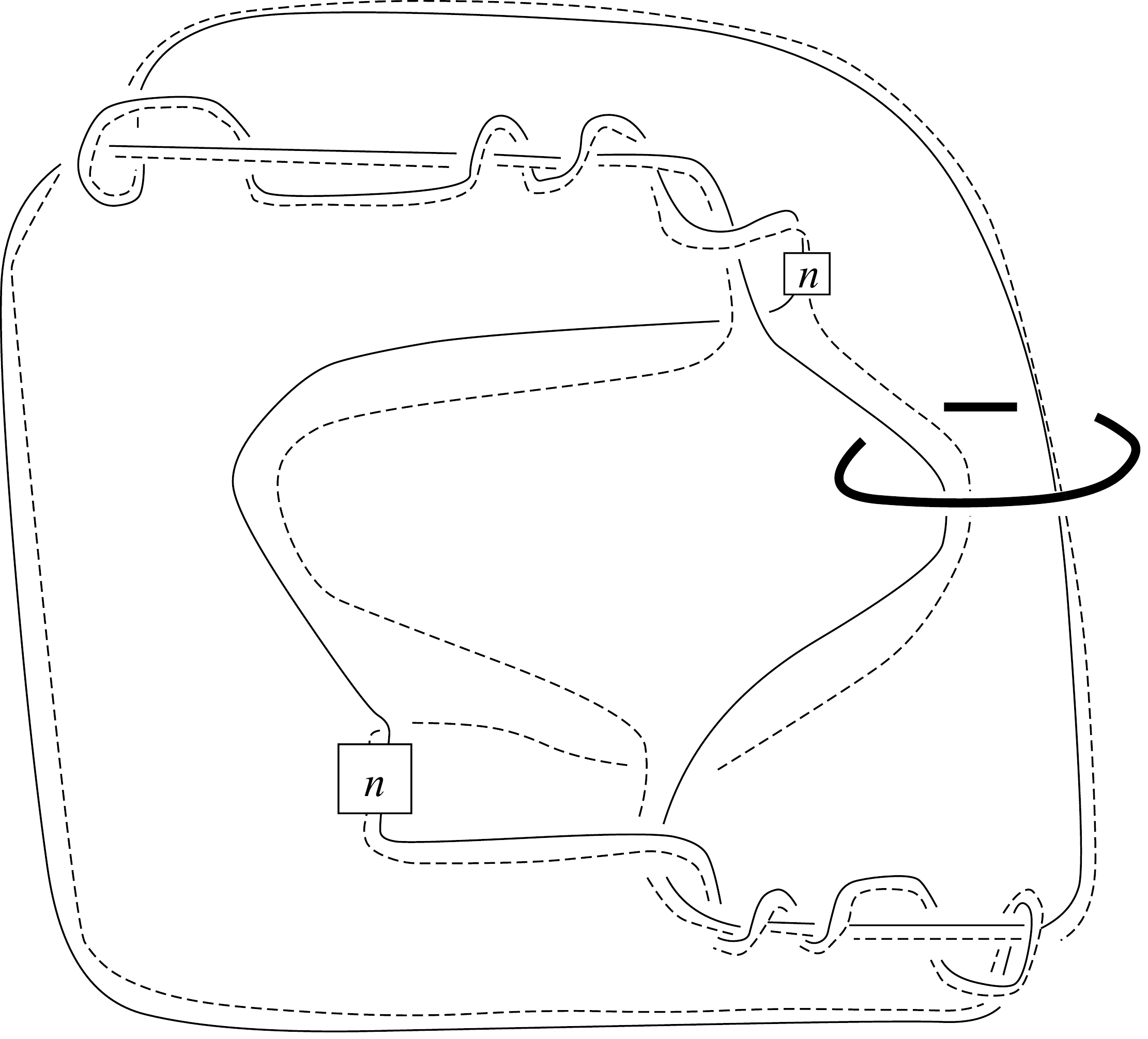}
\end{center}
\vspace*{8pt}
\caption{Two-fold branched cover \label{fig12}}
\end{figure}
Notice that there are $8n+4$ crossings between components: in each half of the diagram there are $2n+1$ external crossings from the twisting and $2n+1$ external crossings from where the second component looped around itself in the base link. If we assign an orientation to this link, we see that all of the external crossings are positive, so linking number is $4n+2$. Therefore, $(0,2n+1)$ is a lower bound for $(0,b)$ unlinkability. 
\end{proof}
\newpage

We now return to our theorem: 
\begin{theorem}
Given a positive integer $d$, there exists a link $L$ that is $(a,0)$-unlinkable, but not $(0,a+d)$-unlinkable for some $a \in \mathbb{N}$. 
\end{theorem}
\begin{proof}
Fix $d$. The $d^{\text{th}}$ member of the $L8a1$ family is $(2,0)$-unlinkable, but not $(0,2d)$ unlinkable (by Lemma 4.2). Since $2d \geq 2+d$ whenever $d \geq 2$, then the $d^{\text{th}}$ member of the $L8a1$ family is not $(0,2+d)$ unlinkable for $d \geq 2$. If $d=1$, the second member of the family is $(2,0)$ unlinkable, but not $(0,3)$ unlinkable.
\end{proof}

\section{Future Research}

One method that may be worth pursuing for the case of changing crossings within both components (for which there is no method presented aside from the consideration of the Conway polynomial) is taking a $\mathbb{Z}_2 \times \mathbb{Z}_2$ cover branched twice over each of the components.  This method seems suited to the situation, and naturally generalizes the approach for unlinking restricted to components.  

We may also consider viewing unlinking as a step-by-step process of changing crossings one at a time, rather than a transition of changing all crossings simultaneously.  When doing so, the most obvious way to unlink often seems to require that the intermediate components of the links will be knotted. For instance, in the $L8a1$ link shown previously, changing the first component crossing in the middle of the second circle knots the first component. We wonder whether there are links that are $(a,0)$-unlinkable, but are not $(a,0)$-unlinkable if we require that the link does not go through knotting as crossings are consecutively changed. 

We are also considering the crossing changes necessary to change links with other linking numbers into a given base link.  Furthermore we could explore links with more than two components.  When doing so, we have the extra obstruction of link homotopy, which would separate into more refined base link classes.

\end{document}